\documentclass[11pt]{amsart}

\usepackage[utf8]{inputenc}
\usepackage{fullpage}
\usepackage{tikz}
\usetikzlibrary{arrows.meta}
\usetikzlibrary{calc}

\usepackage{amsmath}
\usepackage{amsfonts}
\usepackage{amssymb}
\usepackage{amsthm}
\usepackage{amsrefs}

\usepackage{hyperref}

\usepackage{multirow}

\usepackage{graphicx}
\usepackage[all]{xy}

\usepackage{array}

\numberwithin{equation}{section}

\theoremstyle{plain}
\newtheorem{theorem}[equation]{Theorem}

\newtheorem{proposition}[equation]{Proposition}
\newtheorem{lemma}[equation]{Lemma}
\newtheorem{corollary}[equation]{Corollary}

\theoremstyle{definition}
\newtheorem{definition}[equation]{Definition}

\theoremstyle{remark}
\newtheorem{remark}[equation]{Remark}
\newtheorem{example}[subsection]{Example}

\def\defi{\emph}

\def\epsilon{\varepsilon}
\def\theta{\vartheta}
\def\tilde{\widetilde}

\def\Magma{\textsc{Magma}}
\def\Maple{\textsc{Maple}}
\def\SageMath{\textsc{SageMath}}

\def\cf{cf.}

\def\ie{i.e.}

\def\loccit{loc.\ cit.}

\DeclareMathOperator{\Aut}{Aut}

\DeclareMathOperator{\disc}{disc}

\DeclareMathOperator{\End}{End}

\DeclareMathOperator{\GL}{GL}

\DeclareMathOperator{\Hom}{Hom}

\DeclareMathOperator{\Jac}{Jac}

\DeclareMathOperator{\Pic}{Pic}

\DeclareMathOperator{\PSL}{PSL}

\DeclareMathOperator{\Sym}{Sym}

\DeclareMathOperator{\tr}{tr}

\def\C{\mathbb{C}}

\def\F{\mathbb{F}}

\def\P{\mathbb{P}}
\def\Q{\mathbb{Q}}
\def\R{\mathbb{R}}

\def\Z{\mathbb{Z}}

\def\kbar{\overline{k}}

\def\Qbar{\overline{\Q}}

\usepackage{bm}

\DeclareMathOperator{\HH}{H}
\newcommand{\CC}{\mathbb{C}}
\newcommand{\PP}{\mathbb{P}}
\newcommand{\ZZ}{\mathbb{Z}}

\hypersetup{
  pdfauthor   = {},
  pdftitle    = {Numerical computation of endomorphism rings},
  pdfsubject  = {},
  pdfkeywords = {},
  backref=true, pagebackref=true, hyperindex=true, colorlinks=true,
  breaklinks=true, urlcolor=blue, linkcolor=blue, citecolor=blue,
  bookmarks=true, bookmarksopen=true}

\author{Nils Bruin}
\address{Department of Mathematics, Simon Fraser University, Burnaby, BC, Canada V5A 1S6}
\email{nbruin@sfu.ca}
\author{Jeroen Sijsling}
\address{Institut f\"ur Reine Mathematik, Universit\"at Ulm, Helmholtzstrasse 18, 89081 Ulm, Germany}
\email{jeroen.sijsling@uni-ulm.de}
\author{Alexandre Zotine}
\address{Department of Mathematics, Simon Fraser University, Burnaby, BC, Canada V5A 1S6}
\email{szotine@sfu.ca}
\thanks{The research of the first and third author is partially supported by NSERC, and the second author is supported by a Juniorprofessurprogramm of the Science Ministry of Baden-W\"urttemberg.}

\title{Numerical computation of endomorphism rings of Jacobians}
\date{May 28, 2018}
\begin{document}

\begin{abstract}
  We give practical numerical methods to compute the period matrix of a plane
  algebraic curve (not necessarily smooth). We show how automorphisms and
  isomorphisms of such curves, as well as the decomposition of their Jacobians
  up to isogeny, can be calculated heuristically. Particular applications
  include the determination of (generically) non-Galois morphisms between
  curves and the identification of Prym varieties.
\end{abstract}

\maketitle

\section{Introduction}

Let $k$ be a field of characteristic $0$ that is finitely generated over $\Q$.
We choose an embedding of $k$ into $\C$. In this article, we consider
nonsingular, complete, absolutely irreducible algebraic curves $C$ over $k$ of
genus $g$. We represent such a curve $C$ by a possibly singular affine plane
model
\begin{equation}\label{E:curve_eq}
  \tilde{C}\colon f(x,y)=0, \text{ where } f(x,y)\in k[x,y].
\end{equation}
Associated to $C$ is the Jacobian variety $J = \Jac (C)$ representing
$\Pic^0(C)$. Classical results by Abel and Jacobi establish
\begin{equation*}
  J (\C) \cong \HH^0 (C_{\CC}, \Omega_C^1)^* / \HH_1 (C (\CC), \ZZ) \cong \CC^g
  / \Omega \ZZ^{2 g},
\end{equation*}
for a suitable $g\times 2g$ matrix $\Omega$, called a \emph{period matrix} of
$C$.

Let $J_1 = \Jac (C_1)$ and $J_2 = \Jac (C_2)$ be two such Jacobian varieties.
The $\ZZ$-module $\Hom_{\kbar} (J_1, J_2)$ of homomorphisms defined over the
algebraic closure $\kbar$ of $k$ is finitely generated and can be represented
as the group of $\C$-linear maps $\CC^{g_1} \to \CC^{g_2}$ mapping the columns
of $\Omega_1$ into $\Omega_2 \ZZ^{g_2}$. As described in \cite[\S 2.2]{CMSV},
we can heuristically determine homomorphism modules, along with their tangent
representations, from numerical approximations to $\Omega_1,\Omega_2$. These can
then serve as input for rigorous verification as in \loccit

In this article we consider the problem of computing approximations to period
matrices for arbitrary algebraic curves for the purpose of numerically
determining homomorphism modules and endomorphism rings. We also describe how
to identify the (finite) symplectic automorphism groups in these rings, and
with that the automorphism group of the curve. We give several examples of how
the heuristic determination of such objects can be used to obtain rigorous
results.

There is extensive earlier work on computing period matrices for applications
in scientific computing to Riemann theta functions and partial differential
equations. For these applications, approximations that fit in standard machine
precision tend to be sufficient. Number-theoretic applications tend to need
higher accuracy and use arbitrary-precision approximation. Hyperelliptic curves
have received most attention, see for instance Van Wamelen's
\cite{vanWamelen2006} implementation in \Magma. In practice it is limited to
about 2000 digits. Recent work by Molin--Neurohr \cite{molin-neurohr} can reach
higher accuracy and also applies to superelliptic curves.

For general curves, a \Maple\ package based on Deconinck and Van Hoeij
\cite{DeconinckVanHoeij2001} computes period matrices at system precision or
(much more slowly) at arbitrary precision. Swierczewski's reimplementation in
\SageMath\ \cite{swier} only uses machine precision and no high-order numerical
integration. During the writing of this article, another new and fast \Magma\
implementation was developed by Neurohr \cite{neurohr-thesis}. See the
introduction of \cite{neurohr-thesis} for a more comprehensive overview of the
history and recent work on the subject.

Our approach is similar to the references above (in contrast to, for instance,
the deformation approach taken in \cite{sertoz2018}) in that we basically use
the definition of the period matrix to compute an approximation.

\smallskip
\noindent\textbf{Algorithm} Compute approximation to period matrix.

\emph{Input:} $f$ as in \eqref{E:curve_eq}  over a number field and a given
working precision.

\emph{Output:} Approximation of a period matrix of the described curve.

\begin{enumerate}
  \item[1.] Determine generators of the fundamental group of $C$ (Section
    \ref{S:monodromy}).
  \item[2.] Derive a symplectic basis $\left\{ \alpha_1,\ldots\alpha_g,\beta_1 ,
      \ldots ,\beta_{g} \right\}$ of the homology group $H_1 (C (\CC), \ZZ)$
      (Section \ref{S:homology}).
  \item[3.] Determine a basis $\left\{ \omega_1, \dots, \omega_g \right\}$ of
    the space of differentials $\HH^0 (C_{\CC}, \Omega_C^1)$ (Section
    \ref{S:diffbasis}).
  \item[4.] Approximate the period matrix $\Omega=( \int_{\alpha_j} \omega_i,
    \int_{\beta_j} \omega_i)_{i,j}$ using numerical integration (Section
    \ref{S:period_computation}).
\end{enumerate}

We list some notable features of our implementation.
\begin{itemize}
  \item[a.] We use \emph{certified} homotopy continuation \cite{Kranich2016} to
    guarantee that the analytic continuations on which we rely are indeed
    correct. This allows us to guarantee that increasing the working precision
    sufficiently will improve accuracy.
  \item[b.] We base our generators of the fundamental group on a Voronoi cell
    decomposition to obtain paths that stay away from critical points. This is
    advantageous for the numerical integration.
  \item[c.]\label{point:graph_cycle_basis} We determine homotopy generators by
    directly lifting the Voronoi graph to the Riemann surface via analytic
    continuation and taking a cycle basis of that graph. This avoids the
    relatively opaque procedure \cite{TretkoffTretkoff1984} used in
    \cite{DeconinckVanHoeij2001} and \cite{neurohr-thesis}.
  \item[d.] We provide an implementation in a free and open mathematical
    software suite (\SageMath\ version 8.0+), aiding verification of the
    implementation and adaptation and extension of its features.
\end{itemize}

We share the use of Voronoi decompositions with \cite{vanWamelen2006}. This is
no coincidence, since the first author suggested its use to Van Wamelen at the
time, while sharing an office in Sydney, and was eager to see its use tested
for general curves. Dealing with hyperelliptic and superelliptic curves,
\cite{vanWamelen2006} and \cite{molin-neurohr} use a shortcut in determining
homotopy generators. The explicit use of a graph cycle basis in
Step~\ref{point:graph_cycle_basis} above, while directly suggested by basic
topological arguments, is to our knowledge new for an implementation in
arbitrary precision.

The run time of these implementations is in practice dominated by the numerical
integration. The complexity for all these methods is essentially the same, see
\cite[\S4.8]{neurohr-thesis} for an analysis, as well as a fairly systematic
comparison. For a rough idea of performance we give here some timings for the
computation of period matrices of the largest genus curves in each of our
examples. Timings were done using Linux on a Intel i7-2600 CPU at 3.40GHz, at
working precision of 30 decimal digits; 100 binary digits.
\begin{center}
  \begin{tabular}{c|c|c}
    Curve&\Maple\ 2018&\SageMath\ 8.3-$\beta0$\\
    \hline
    $C$ from Example~\ref{E:genus6}&99.6 sec&45.5 sec\\
    $C$ from Example~\ref{E:macbeath}&133.2 sec&8.59 sec\\
    $D$ from Example~\ref{E:prym}&119.2 sec&12.8 sec\\
  \end{tabular}
\end{center}
With recent work on rigorous numerical integration \cite{johansson2018}, which
is now also available in \SageMath, it would be possible to modify the program
to return certified results. While this is worthwhile and part of future work,
rigorous error bounds would make little difference for our applications, since
we have no \emph{a priori} height bound on the rational numbers we are trying
to recognize from floating point approximations. One of our objectives is to
provide input for the rigorous verification procedures described in
\cite{CMSV}.

Our main application is to find decompositions of $\Jac(C)$ via its
endomorphism ring $\End_{\kbar}(J)=\Hom_{\kbar} (J,J)$. Idempotents of
$\End(J)$ give rise to isogenies to products of lower-dimensional abelian
varieties \cite[Ch.~5]{birkenhake-lange}, \cite{KaniRosen1989}. Furthermore,
since $\End(J)$ has a natural linear action on $\HH^0 (C, \Omega_C^1)^*$,
idempotents induce projections from the canonical model of $C$. For composition
factors arising from a cover $\phi \colon C \to D$, the corresponding
projection factors through $\phi$, so we can recover $\phi$ from it. In the
process, we verify $\phi$ rigorously, as well as the numerically determined
idempotent.

Finally, having determined $\End (J)$, we can compute the finite group
automorphisms of $J$ that are fixed by the Rosati involution. Its action on
$\HH^0(C,\Omega_C^1)$ gives, via the Torelli Theorem
\cite[Theorem~12.1]{Milne1986}, a representation of the automorphism group
$\Aut (C) = \Aut_{\kbar} (C)$ of $C$ on a canonical model. There are other
approaches to computing automorphism groups of curves, for instance
\cite{Hess2004}. The approach described here naturally finds a candidate for
the \emph{geometric} automorphism group (members of which are readily
rigorously verified to give automorphisms) whereas more algebraically oriented
approaches, such as the one in \cite{Hess2004}, tend only to find the
automorphisms defined over a given base field or have prohibitive general
running times. We describe the corresponding algorithm in
Section~\ref{sec:isos}.

These results are applied to numerically identify some Prym varieties in higher
genus. In particular, we find isogeny factors $\Jac (D)$ of Jacobians $\Jac
(C)$ that do not come from any morphism $C \to D$, or come from a morphism that
is not a quotient by automorphisms of $C$.

\textbf{Acknowledgments} We would like to that Catherine Ray for pointing out
some errors in a previous version of Section~\ref{sec:homs}.

\section{Computation of homology}

We compute a homology basis for $C (\CC)$ from its fundamental group. We obtain
generators for this group by pulling back generators of of the fundamental
group of a suitably punctured Riemann sphere covered by $C$. Such pullbacks can
be found by determining the analytic continuations of appropriate algebraic
functions. In order to make these continuations amenable to computation, we use
paths that stay away from any ramification points.

The function $x$ on $\tilde{C}$ induces a morphism $x\colon C\to\PP^1$ and
therefore expresses $C$ as a finite (ramified) cover of $\PP^1$ of degree $n$
say. We collect terms with respect to $y$ and write
\begin{equation*}
  f(x,y) = f_n(x) y^n + f_{n-1}(x) y^{n-1} + \cdots + f_0(x),
\end{equation*}
where $f_0(x),\ldots,f_n(x)\in k[x]$, with $f_n(x)\neq 0$. We write
$\PP^1(\CC)=\CC\cup\{\infty\}$, and define the \emph{finite critical locus} of
$x$ as
\begin{equation*}
  S = \{ x \in \CC : \disc_y(f)(x) = 0 \} .
\end{equation*}
We set $S_\infty=S\cup\{\infty\}$, so that $x$ induces an unramified cover
$C-x^{-1}(S_\infty)$ of $\CC-S$.

\subsection{Fundamental group of $\CC-S$}

\begin{figure}
  \centerline{\includegraphics{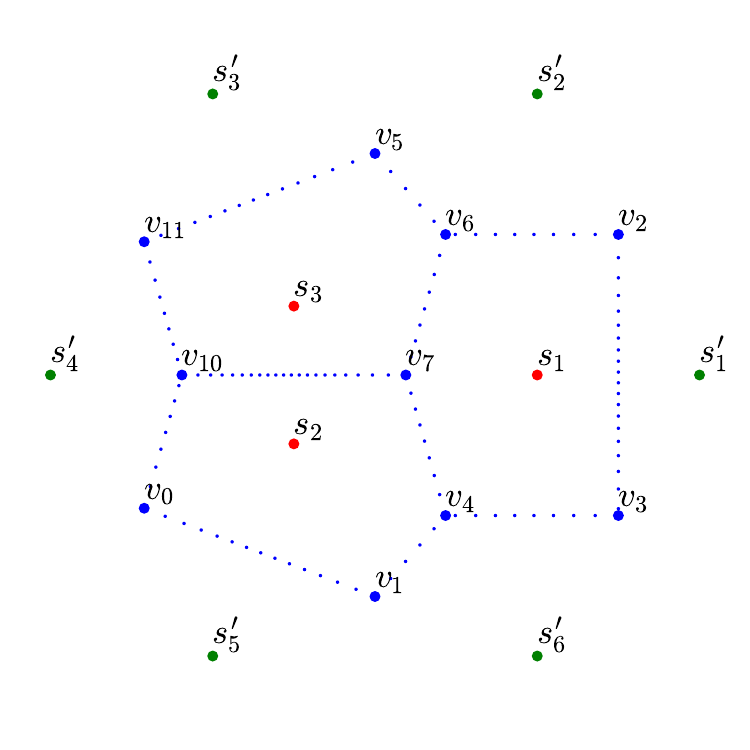}}
  \vspace{-2em}
  \caption{Paths for $C\colon y^2=x^3-x-1$. The dots marking the edges indicate
    the step size used for the certified homotopy continuation.}
  \label{figure:homology}
\end{figure}

We describe generators of the fundamental group of $\CC-S$ by cycles in a planar
graph that we build in the following way.

We approximate the circle with centre $c_0=\frac{1}{\#S}\sum_{s\in S}s$ and
radius $2\max_{s\in S} | s-c_0|$ using a regular polygon with vertices, say,
$s_1',\ldots,s_6'$. Then we compute the Voronoi cell decomposition (see
e.g.~\cite{Aurenhammer1991}) of $\CC$ with respect to
$S'=S\cup\{s_1',\ldots,s_6'\}$. This produces a finite set of vertices
$V=\{v_1,\ldots,v_r\}\subset\CC$ and a set $E$ of line segments $e_{ij}$
between $v_i,v_j\in V$ such that the regions
\begin{equation*}
  F_{s}=\{x\in \CC: |x-s|\leq|x-s'| \text{ for any }s'\in S'-\{s\} \}
\end{equation*}
have boundaries consisting of $e_{ij}$, together with some rays for unbounded
regions. We define $F_\infty=\bigcup_{s\in S'-S} F_s$. Then we see that $F_s$
for $s\in S_\infty$ has a finite boundary, giving a loop separating $s$ from
the rest of $S_ \infty$. See Figure~\ref{figure:homology} for an illustration
of the resulting graph for the curve $C\colon y^2=x^3-x-1$. It illustrates the
set $S={s_1,s_2,s_3}$, together with the additional points $s'_1,\ldots,s'_6$,
and the vertices $v_0,\ldots,v_{11}$ and edges between them, bounding the
Voronoi cells $F_s$.

\begin{lemma}
  \begin{enumerate}
    \item The boundaries of the regions $F_s$ for $s\in S_\infty$ provide
      cycles that generate $\HH^1(\CC-S,\Z)$.
    \item The fundamental group $\pi_1(\CC-S,v_i)$ is generated by cycles in
      the graph $(V,E)$.
  \end{enumerate}
\end{lemma}

\begin{proof}
  The first claim follows because the boundaries exactly form loops around each
  individual point $s$. The second claim follows because the graph is
  connected. Hence, we can find paths that begin and end in $v_i$ and (because
  of the first claim) provide a simple loop around a point $s\in S_\infty$.
\end{proof}

\subsection{Lifting the graph via homotopy continuation}
\label{S:homotopy_continuation}

Each of our vertices $v_i\in \CC$ has exactly $n$ preimages $v_i^{(1)}, \ldots,
v_i^{(n)}$, determined by the $n$ distinct simple roots of the equation
$f(v_i,y)=0$. We can parametrize each edge $e_{ij}$ from our graph by
$x(t)=(1-t)v_i+tv_j\text{ for } t\in [0,1]$. We lift $e_{ij}$ to paths
$e_{ij}^{(1)},\ldots,e_{ij}^{(n)}$ using the branches $y^{(k)}(t)$ defined by
\begin{equation*}
  f(x(t),y^{(k)}(t))=0\text{ and }y^{(k)}(0)=y(v_i^{(k)}).
\end{equation*} Since $e_{ij}$
stays away from the critical locus, the function $y^{(k)}(t)$ is well-defined
by continuity. Moreover, it is analytic in a neighbourhood of $e_{ij}^{(k)}$.

Given $k$, we have $y^{(k)}(1)=v^{k'}_j$ for some $k'$. Hence, every edge
$e_{ij}$ determines a permutation $\sigma_{ij}$ such that $\sigma_{ij}(k)=k'$.
The lifted edge $e^{(k)}_{ij}$ connects $v_i^{(k)}$ to
$v_k^{(\sigma_{ij}(k))}$. We write $(V',E')$ for this lifted graph on $C(\CC)$.
If we split up the path in sufficiently small steps, we can determine these
permutations.

\begin{lemma} \label{L:certified_steps}
  With the notation above and for given $i,j$, we can algorithmically determine a
  subdivision
  \begin{equation*}
    0=t_0<t_1<t_2<\cdots<t_{m_{ij}}=1
  \end{equation*}
  and real numbers $\epsilon_0,\ldots,\epsilon_{m_{ij}-1}$ such that for $t,m$
  satisfying  $t_m\leq t\leq t_{m+1}$, we have that
  $|y^{(k')}(t)-y^{(k)}(t_m)|<\epsilon_m$  if and only if $k'=k$.
\end{lemma}

\begin{proof}
  We construct the $t_m,\epsilon_m$ iteratively, starting with $m=0$. We set
  \begin{equation*}
    \epsilon_m = \frac{1}{3} \min \{|y^{(k_1)}(t_m) - y^{(k_2)}(t_m)|: k_1\neq k_2\}
  \end{equation*}
  Using \cite[Theorem~2.1]{Kranich2016}, we can determine from $f(x(t),y)$,
  $\epsilon$, and $x(t_m)$ a value $\delta>0$ such that for values $t$
  satisfying $t_m\leq t\leq t_m+\delta$ we have that $|y^{(k)}(t) -
  y^{(k)}(t_m)| < \epsilon_m$. It follows that we can set
  $t_{m+1}=\min(1,t_m+\delta)$. Inspection of the formulas for $\delta$ give us
  that if the distance of any critical point from the path is positive, then
  there is a finite $m$ such that $t_m=1$.
\end{proof}

\begin{remark}
  In Figure~\ref{figure:homology}, the dots on the edges mark the sequence
  $x(t_0),x(t_1),\ldots,x(t_{m_{ij}})$. In particular, on the edge from $v_7$
  to $v_{10}$ one can see that as the distance to the branch points $s_2,s_3$
  gets smaller, the step sizes are reduced accordingly.
\end{remark}

\begin{lemma}\label{L:newton_step}
  Given $\epsilon<\epsilon_m$, $t$ with $t_m<t\leq t_{m+1}$, and $\tilde{y}_m$
  with $|\tilde{y}_m-y^{(k)}(t_m)|<\epsilon$, we can use Newton iteration to
  compute $\tilde{y}$ such that $|\tilde{y}-y^{(k)}(t)|<\epsilon$.
\end{lemma}

\begin{proof}
  We use Newton iteration to approximate a root of $f(t,y)$, with initial value
  $\tilde{y}_m$. We are looking for the unique root that lies within a radius
  of $\epsilon_{m}$ of the initial value. If at any point the Newton iteration
  process escapes this disk, or if the iteration does not converge sufficiently
  quickly, we insert the point $(t_m+t)/2$ and restart. We know that if Newton
  iteration converges to a value in the disk, it must be the correct value.
  Furthermore, continuity implies that convergence will occur if if $|t - t_m|$
  is small enough.
\end{proof}

Since $x(t_0)\notin S$ we can use standard complex root finding algorithms on
$f(x(t_0),y)=0$, to find approximations $\tilde{y}^{(k)}_0$ to any desired
finite accuracy. We then use Lemma~\ref{L:newton_step} iteratively to find an
approximation $\tilde{y}^{(k)}_m$ to $y^{(k)}(t_m)$, for each
$m=1,\ldots,m_{ij}$.

The Voronoi graph $(V,E)$ generates the fundamental group of $\CC-S$, so the
lifted graph $(V',E')$ generates the fundamental group of the unramified cover
$C(\CC) - x^{-1}(S_\infty)$, and therefore also of $C(\CC)$. We have assumed
that $C$ is an absolutely irreducible algebraic curve, so the graph is
connected.

\begin{remark}\label{R:stored_continuation_information}
  For computing integrals along $v_{ij}^{(k)}$ in
  Section~\ref{S:period_computation}, we store for each relevant edge $e_{ij}$
  the vectors $\{(t_m,\epsilon_m,\tilde{y}^{(1)}_m,\ldots,\tilde{y}^{(n)}_m):m
  \in\{0,\ldots,m_{ij}\}\}$. With this information we can quickly, reliably,
  and accurately approximate $y^{(k)}(t)$ for $t\in[0,1]$ using
  Lemma~\ref{L:newton_step}.
\end{remark}

\subsection{Computing the monodromy of $C\to\PP^1$}
\label{S:monodromy}

We do not need this in the rest of the paper, but a side effect of computing
the lifted graph is that we can also compute the monodromy of the cover $C \to
\PP^1$. To any path in the Voronoi graph we associate a permutation by
composing the permutations associated with the constituent edges. For example,
to the path $p=(v_1,v_2,v_3)$ we associate the permutation $\sigma_p =
\sigma_{12}\sigma_{23}$ (assuming that our permutations act on the right).
Choosing, say, $v_1$ as our base point, this provides us with a group
homomorphism $\pi_1(\CC-S,v_1) \to \Sym(n)$. The image gives the group of deck
transformations of the cover or, in terms of field theory, a geometric
realization of the Galois group of the degree $n$ field extension of $\CC(x)$
given by $\CC(x)[y]/(f(x,y))$. In particular, by taking a path that forms a
loop around a single point $s\in C\cup\{\infty\}$, we can obtain the local
monodromy of $s$. The cycle type of the corresponding permutation gives the
ramification indices of the fibre over $s$. In particular, if the permutation
is trivial, then $C\to \PP^1$ is unramified over $s$.

\subsection{Symplectic homology basis}
\label{S:homology}

\begin{figure}
  \begin{center}
    \begin{tikzpicture}[scale = 1]
      \clip (-10.5,-2.5) rectangle (-4,2.5);
      \coordinate (v0) at (-7,0); \node [above right] at (v0) {$v_0^{(1)}$};
      \coordinate (v1) at (-7.6,-1); \node [above left] at (v1) {$v_1^{(1)}$};
      \coordinate (v2) at (-7.3, 1.4); \node [above right] at (v2) {$v_2^{(1)} = v_3^{(1)}$};
      \coordinate (v4) at (-6, -0.5); \node [above right] at (v4) {$v_4^{(1)}$};
      \foreach \x in {(v0),(v1),(v2),(v4)}{
              \node[draw,circle,inner sep=1pt,fill] at \x {};
      }
      \draw [thick] (v1) -- (v0);
      \draw [thick, ->] (v1) -- ($(v1)!0.75!(v0)$);
      \draw [thick] (v0) -- (v4);
      \draw [thick, ->>] (v0) -- ($(v0)!0.75!(v4)$);
      \draw [thick] (v0) -- (v2);
      \draw [thick, ->] (v0) -- ($(v0)!0.75!(v2)$);
      \draw [thick, ->>] (v2) -- ($(v2)!0.75!(v0)$);
      \draw [thick, dotted] (v1) -- (-8.2, -1.2);
      \node [left] at (-8.2, -1.2) {$\alpha$};
      \draw [thick, dotted] (v2) -- (-8.2, 1.6);
      \draw [thick, dotted] (v2) -- (-7.2, 2);
      \draw [thick, dotted] (v4) -- (-5.3, -0.5);
      \node [right] at (-5.3, -0.5) {$\beta$};
      \node at (-7, -2) {$\langle \alpha, \beta \rangle_{v_0^{(1)}}^{\text{in}} = 0$, ~
              $\langle \alpha, \beta \rangle_{v_0^{(1)}}^{\text{out}} = -\frac{1}{2}$};
	\end{tikzpicture}
    \begin{tikzpicture}[scale = 1]
      \clip (-4,-2.5) rectangle (3,2.5);
      \coordinate (v0) at (-1,0); \node [above right] at (v0) {$v_0^{(1)}$};
      \coordinate (v1) at (-1.6,-1); \node [above left] at (v1) {$v_1^{(1)}$};
      \coordinate (v2) at (-0.7, 1.4); \node [right] at (v2) {$v_2^{(1)}$};
      \coordinate (v3) at (-1.7, 1); \node [left] at (v3) {$v_3^{(1)}$};
      \coordinate (v4) at (0, -0.5); \node [above right] at (v4) {$v_4^{(1)}$};
      \foreach \x in {(v0),(v1),(v2),(v3),(v4)}{
              \node[draw,circle,inner sep=1pt,fill] at \x {};
      }
      \draw [thick] (v1) -- (v0);
      \draw [thick, ->] (v1) -- ($(v1)!0.75!(v0)$);
      \draw [thick] (v0) -- (v4);
      \draw [thick, ->>] (v0) -- ($(v0)!0.75!(v4)$);
      \draw [thick] (v0) -- (v2);
      \draw [thick, ->] (v0) -- ($(v0)!0.75!(v2)$);
      \draw [thick] (v3) -- (v0);
      \draw [thick, ->>] (v3) -- ($(v3)!0.75!(v0)$);
      \draw [thick, dotted] (v1) -- (-2.2, -1.2);
      \node [left] at (-2.2, -1.2) {$\alpha$};
      \draw [thick, dotted] (v2) -- (-0.7, 2);
      \draw [thick, dotted] (v3) -- (-1.9, 1.6);
      \draw [thick, dotted] (v4) -- (0.7, -0.5);
      \node [right] at (0.7, -0.5) {$\beta$};
      \node at (-1, -2) {$\langle \alpha, \beta \rangle_{v_0^{(1)}}^{\text{in}} = -\frac{1}{2}$, ~
              $\langle \alpha, \beta \rangle_{v_0^{(1)}}^{\text{out}} = -\frac{1}{2}$};
    \end{tikzpicture}
    \end{center}
    \vspace{-1em}
  \caption{Examples of the intersection pairing}
\end{figure}
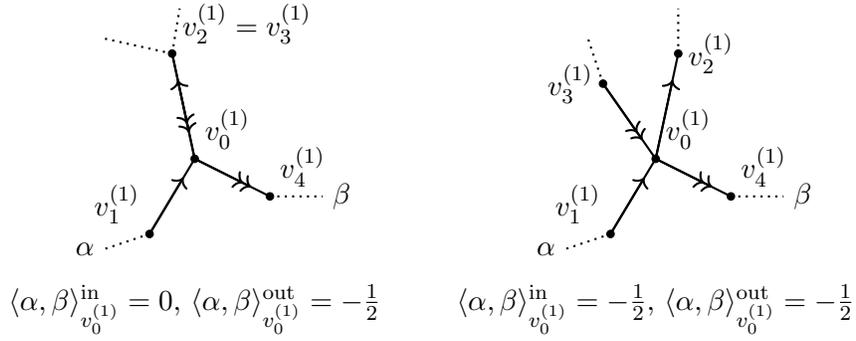

Since $C(\CC)$ is a Riemann surface, it is orientable and hence we have a
symplectic structure on its first homology. The pairing on cycles can be
computed in the following way. Suppose that $\alpha,\beta$ are two paths
intersecting at $v_0^{(1)}$, and that $\alpha$ contains the segment
$v_1^{(1)}\to v_0^{(1)}\to v_2^{(1)}$ and that $\beta$ contains the segment
$v_3^{(1)}\to v_0^{(1)}\to v_4^{(1)}$. We define
\begin{equation*}
  \langle \alpha,\beta\rangle_{v_0^{(1)}} =
  \langle \alpha,\beta\rangle_{v_0^{(1)}}^{\text{in}} +
  \langle \alpha,\beta\rangle_{v_0^{(1)}}^{\text{out}},
\text{ and }
\langle \alpha,\beta\rangle=\sum_v \langle \alpha,\beta \rangle_v,
\end{equation*}
where $\langle\alpha,\beta\rangle_{v_0^{(1)}}=0$ if $\alpha$ or $\beta$ do not pass through $v_0^{(1)}$, and otherwise
\begin{equation*}
  \begin{aligned}
    \langle \alpha,\beta\rangle_{v_0^{(1)}}^{\text{in}} = &
    \begin{cases}
      0&\text{ if }v_3=v_1\text{ or }v_3=v_2\\
      \frac{1}{2}&\text{ if $v_1,v_3,v_2$ are counterclockwise oriented around $v_0$}\\
      -\frac{1}{2}&\text{ if $v_1,v_3,v_2$ are clockwise oriented around $v_0$,}
    \end{cases} \\
    \langle \alpha,\beta\rangle_{v_0^{(1)}}^{\text{out}} = &
    \begin{cases}
      0&\text{ if }v_3=v_1\text{ or }v_4=v_2\\
      \frac{1}{2}&\text{ if $v_1,v_2,v_4$ are counterclockwise oriented around $v_0$}\\
      -\frac{1}{2}&\text{ if $v_1,v_2,v_4$ are clockwise oriented around $v_0$}.
    \end{cases}
  \end{aligned}
\end{equation*}

At vertices where $\alpha,\beta$ meet transversely, this is clearly the usual
intersection pairing on $H_1(C(\CC),\ZZ)$. A deformation argument verifies that
the half-integer weights extend it properly to cycles with edges in common.

\begin{lemma}
  By applying an algorithm by Frobenius \cite[\S7]{Frobenius1879} we can find
  a $\ZZ$-basis $\alpha_1,\ldots,\alpha_g,\beta_1,\ldots,\beta_g$ for
  $\HH_1(C(\CC),\ZZ)$ such that $\langle \alpha_i,\alpha_j\rangle=\langle
  \beta_i,\beta_j\rangle=0$ and $\langle \alpha_i,\beta_j\rangle=\delta_{ij}$.
\end{lemma}

\begin{proof}
  We first compute a cycle basis for the lifted graph $(V',E')$ described in
  Section~\ref{S:homotopy_continuation}, say $\gamma_1,\ldots,\gamma_r$ and
  compute the antisymmetric Gram matrix $G_\gamma=(\langle \gamma_i,
  \gamma_j\rangle)_{ij}$. Frobenius's algorithm yields an integral
  transformation $B$ such that $BG_\gamma B^T$ is in symplectic normal form,
  \ie, a block diagonal matrix with $g$ blocks
  \begin{equation*}
    \begin{pmatrix}0&d_i\\-d_i&0\end{pmatrix},
  \end{equation*}
  possibly followed by zeros, with $d_1\mid d_2\mid \cdots \mid d_g$. Because
  $C(\CC)$ is a complete Riemann surface, we know that $d_1=\cdots=d_g=1$ and
  that $g$ is the genus of $C(\CC)$. The matrix $B$ gives us
  $\alpha_1,\beta_1,\ldots,\alpha_g,\beta_g$ as $\ZZ$-linear combinations of
  our initial cycle basis $\gamma_1,\ldots,\gamma_r$.
\end{proof}

\section{Computing the period lattice}

\subsection{A basis for $\HH^0(C,\Omega_C^1)$}
\label{S:diffbasis}

From the adjunction formula \cite{ACGH1985} we know that $\HH^0(C,\Omega_C^1)$
is naturally a subspace of the span of
\begin{equation*}
  \left\{\frac{h\,dx}{\partial_yf(x,y)}: h=x^iy^j
  \text{ with }0\leq i,j
  \text{ and }i+j\leq n-3\right\}.
\end{equation*}
If the projective closure of $\tilde{C}$ is nonsingular, then
$\HH^0(C,\Omega_C^1)$ is exactly this span. If $\tilde{C}$ has only
singularities at the projective points $(1:0:0),(0:1:0),(0:0:1)$ then Baker's
theorem \cite{Baker1893} states that we can take those $(i,j)$ for which
$(i+1,j+1)$ is an interior point to the Newton polygon of $f(x,y)$. In even
more general situations, the adjoint ideal \cite[A\S2]{ACGH1985} specifies
exactly which subspace of polynomials $g$ corresponds to the regular
differentials on $C$. We use Baker's theorem when it applies and otherwise rely
on Singular~\cite{DGPS} to provide us with a basis
\begin{equation*}
  \left\{\omega_i = \frac{h_i dx}{\partial_yf(x,y)}:
          i=1,\ldots,g \right\} \subset \HH^0(C,\Omega_C^1).
\end{equation*}

\subsection{Computing the period matrix}
\label{S:period_computation}

Given a basis $\omega_1,\ldots,\omega_g$ for $\HH^0(C,\Omega_C^1)$ and a
symplectic basis $\alpha_1,\ldots,\alpha_g,\beta_1,\ldots,\beta_g$ for
$\HH_1(C(\CC),\ZZ)$, the corresponding \emph{period matrix} is
\begin{equation*}
  \Omega_{\alpha\beta} = \left( \Omega_\alpha | \Omega_\beta \right)
  = \left( \int_{\alpha_j}\omega_i \,\right|\,\left. \int_{\beta_j}\omega_i\right)_{ij}.
\end{equation*}
The resulting \emph{period lattice} is the $\ZZ$-span
$\Lambda=\Omega_{\alpha\beta}\ZZ^{2g}$ of the columns in $\CC^g$. As an
analytic space, the Jacobian of $C$ is isomorphic to the complex torus
$\CC^g/\Lambda$. Our paths consist of lifted line segments, so we numerically
approximate the integrals along the edges $e_{ij}^{(k)}$ that occur in our
symplectic basis and compute $\Omega_{\alpha\beta}$ by taking the appropriate
$\ZZ$-linear combinations of these approximations. To lighten notation we
describe the process for the edge $e_{12}^{(1)}$. As in
Section~\ref{S:homotopy_continuation} we parametrize the edge by
\begin{equation*}
  x(t) = (1-t)v_1+tv_2
\end{equation*}
and with the stored information (see
Remark~\ref{R:stored_continuation_information}), we can quickly compute
$y^{(k)}(t)$ for given values of $t$. We obtain
\begin{equation*}
  \int_{e_{12}^{(1)}} \omega_i =
  (v_2-v_1) \int_{t=0}^1 \frac{h_i(x(t)y(t))}{\partial_y f(x(t),y(t))}dt.
\end{equation*}
Note that our integrand is holomorphic, so well suited for high order
integration schemes such as Gauss-Legendre and Clenshaw-Curtis. We implemented
Gauss-Legendre with relatively naive node computation. While in our experiments
this was sufficient, there is the theoretical drawback that for very high order
approximations, the determination of the integration nodes becomes the dominant
part. There are sophisticated methods for obtaining the nodes with a better
complexity (see \cite{bogaert}). Alternatively, quadrature schemes like
Clenshaw-Curtis may need more evaluation nodes to obtain the same accuracy, but
allow for faster computation of these nodes.

Rather than compute guaranteed bounds, we have settled on a standard error
estimation scheme, as described in, for instance,
\cite[Section~5]{BaileyEA2005} to adapt the number of evaluation nodes. Since
our applications will not provide proven results anyway, this is sufficient for
our purposes.

\begin{remark}
  \label{R:riemann_matrix}
  There is a split in literature on how to order the symplectic basis for the
  period matrix. With the normalization we use, one gets that
  \begin{equation*}
    \Omega_{\alpha}^{-1}\Omega_{\alpha\beta}=(1\,|\,\Omega_{\alpha}^{-1}\Omega_{\beta})=(1\,|\,\tau)
  \end{equation*}
  where $\tau$ is a Riemann matrix, \ie, a symmetric matrix with positive
  definite imaginary part. Here $\tau$ represents the corresponding lattice in
  Siegel upper half space. In \cite{birkenhake-lange}, the period matrix is
  taken to be $\Omega_{\beta\alpha}$.
\end{remark}

\section{Homomorphism and isomorphism computations}\label{sec:homs}

\subsection{Computing homomorphisms between complex tori}

Let $C_1$ and $C_2$ be two curves with Jacobians $J_1$ and $J_2$. Let
$\Omega_1,\Omega_2$ be period matrices such that $J_1 (\C) = \C^{g_1} /
\Omega_1\Z^{2g_1}$ and $J_2 (\C) = \C^{g_2}/ \Omega_2\Z^{2g_2}$ as analytic
groups.

A homomorphism $\phi\colon J_1\to J_2$ induces a tangent map  $\HH^0 (C_1,
\Omega_{C_1}^1)^* \to \HH^0 (C_2, \Omega^1_{C_2})^*$ and a map on homology
$\HH_1 (C_1, \Z) \to \HH_1 (C_2, \Z)$. After a choice of bases, these
correspond to matrices $T=T_\phi\in M_{g_2,g_1}(\CC)$ and $R=R_\phi\in
M_{2g_2,2g_1}(\ZZ)$, which we call the \emph{tangent representation} and the
\emph{homology representation} of $\phi$.

\begin{proposition}\label{prop:homos}
  Let $\phi\colon J_1\to J_2$ be a homomorphism and let $T$, $R$ be the induced
  matrices described above.
  \begin{enumerate}
    \item The matrices $T=T_\phi$ and $R=R_\phi$ satisfy $T \Omega_1 = \Omega_2 R$.
    \item A pair $(T, R)$ as in (i) comes from a uniquely determined
      homomorphism $\phi\colon J_1 \to J_2$.
    \item Either of the elements $T$ and $R$ in (i) is determined by
      the other.
    \item If the curves $C_1$ and $C_2$ as well as the chosen bases of
      differentials and $\phi$ are defined over $k\subset\Qbar$, then the
      matrix $T$ is an element of $M_{g_2, g_1} (k)$.
  \end{enumerate}
\end{proposition}

\begin{proof}
  These results are in \cite[\S1.2]{birkenhake-lange}. Writing
  $\overline{\Omega}_2$ for the element-wise complex conjugate of $\Omega_{2}$,
  we remark for part (iii) that we can determine $R$ from $T$ by considering
  \begin{equation}\label{eq:solveforR}
    \begin{pmatrix} T \Omega_1 \\ \overline{T \Omega}_1 \end{pmatrix}
    = \begin{pmatrix} \Omega_2 \\ \overline{\Omega}_2 \end{pmatrix} R,
  \end{equation}
  since the first matrix on the right hand side of \eqref{eq:solveforR} is
  invertible. Conversely, we can determine $T$ from $R$ by considering the
  first $g_1$ columns on either side of $T \Omega_1 = \Omega_2 R$ since the
  corresponding matrices are invertible.
\end{proof}

We seek to recover these pairs $(T, R)$ numerically. This question was briefly
touched upon in \cite[6.1]{ants-database}, and before that in \cite[\S
3]{vanWamelen2006}, but here we give some more detail.

\begin{lemma}\label{L:homcomp}
  Given approximations of $\Omega_1,\Omega_2$ to sufficiently high precision,
  we can numerically recover a $\ZZ$-basis for $\Hom(J_1, J_2)$, represented
  by matrices $R\in M_{2g_2,2g_1}(\ZZ)$ and $T\in M_{g_2, g_1} (\C)$ as in
  Proposition \ref{prop:homos}.
\end{lemma}

\begin{proof}
  Following Remark~\ref{R:riemann_matrix}, we can normalize $\Omega_i$ to be of
  the form $(1\,|\,\tau_i)$. We write
  \begin{equation*}
    R = \begin{pmatrix} D & B \\ C & A \end{pmatrix}
    \text{, where }
    D,B,C,A\in M_{g_2,g_1}(\ZZ).
  \end{equation*}
  Then $T = D + \tau_2 C$, and
  \begin{equation*}
    B + \tau_2 A = (D + \tau_2 C) \tau_1 .
  \end{equation*}
  Considering real and imaginary parts separately, we obtain $m=2g_1g_2$
  equations with real coefficients in $n=4g_1g_2$ integer variables, denoted by
  $M \in M_{m,n} (\R)$. We recognize integer solutions that are small compared
  to the precision to which we calculated $\tau_1, \tau_2$ in the following
  way. Observe that such solutions correspond to short vectors in the lattice
  generated by the columns of $(I \mid \epsilon^{-1} M)$, where $\epsilon$ is
  some small real number. The LLL algorithm can find such vectors, and we keep
  the ones that lie in the kernel to the specified precision.

  If sufficient precision is used, then we obtain a basis for $\Hom(J_1, J_2)$
  in this way. (Heuristically, any approximation to high precision will do.)
  Proposition \ref{prop:homos} shows how to recover the tangent representation
  $T$ from the corresponding homology representations $R$.
\end{proof}

\begin{remark}
  An important tuning parameter for applications of LLL is the precision. We
  have an (estimated) accuracy of the entries in the matrix $M$. We choose
  $\epsilon$ such that $\epsilon^{-1}M$ has accuracy to within $0.5$. If we
  have computed the period matrices to a precision of $b$ bits, then $M$
  contains about $2g_1g_2b$ bits of information. We would therefore expect that
  the entries in the LLL basis have entries of size about $(2g_1g_2b)/(4g_1g_2)
  = b/2$. We only keep vectors that have entries of bit-size at most half that.
\end{remark}

In the context of Proposition~\ref{prop:homos}(iv), the algebraic entries of
$T$ can be recognized by another application of the LLL algorithm; for example,
the \SageMath\ implementation\\\verb|number_field_elements_from_algebraics| can
be used to this end. We emphasize that in order to recover this algebraicity,
we need the original period matrices $\Omega_i$ with respect to a basis of
$H^0(C,\Omega^1_C)$ defined over $\Qbar$. A differential basis for which the
period matrix takes the shape $(1\,|\,\tau_i)$ usually has a transcendental
field of definition.

For a Jacobian $J$, the natural principal polarization gives rise to the
Rosati-involution on $\End(J)$ (\cf\ \cite[\S 5.1]{birkenhake-lange}). We
choose a symplectic basis for $H_1(C(\CC),\ZZ)$ and denote the standard
symplectic form by
\begin{equation*}
  E=\begin{pmatrix}
  0&I\\-I&0
  \end{pmatrix}\in M_{2g,2g}(\ZZ).
\end{equation*}

\begin{proposition}
  Let $\phi\colon J \to J$ be an endomorphism with corresponding pair $(T, R)$
  as in Proposition~\ref{prop:homos}(i). Then the Rosati involution
  $\phi^{\dagger}$ of $\phi$ corresponds to the pair $(T^{\dagger},
  R^{\dagger})$ with
  \begin{equation*}
    R^{\dagger} = -E R^t E .
  \end{equation*}
\end{proposition}

\begin{proof}
  Since we chose our homology basis to be symplectic, the Rosati involution of
  the endomorphism corresponding to $R$ corresponds to the adjoint with respect
  to the pairing defined by $E$, which is $E R^t E^{-1} = - E R^t E$.
\end{proof}

\begin{remark}
  Proposition \ref{prop:homos}(iii) shows how to obtain $T^{\dagger}$ from
  $R^{\dagger}$.
\end{remark}

Recall \cite[Chapter 5]{birkenhake-lange} that any polarized abelian variety
allows a decomposition up to isogeny
\begin{equation}\label{eq:decomp}
  J \sim \prod_i A_i^{e_i}
\end{equation}
into powers of simple polarized quotient abelian varieties $A_i$.

\begin{corollary}\label{cor:factors}
  Let $J$ be the Jacobian of a curve $C$, and let $\Omega$ be a corresponding
  period matrix. If we know $\Omega$ to sufficiently high precision, then we
  can numerically determine the factors in \eqref{eq:decomp}. Furthermore, if
  $J$ is defined over $\Qbar$, we can numerically determine a field of
  definition for each of the conjectural factors $A_i$.
\end{corollary}

\begin{proof}
  Using Lemma~\ref{L:homcomp} we can compute generators for $\End(J)$. We can
  then determine symmetric idempotent matrices $e \in M_{2 g, 2 g} (\Q)$ by
  using meataxe algorithms, or alternatively by directly solving $e^2 = e$ in
  the subring of $\End (J)$ fixed by the Rosati involution. The columns of
  $\Omega e$ span a complex torus of smaller dimension. By \cite{KaniRosen1989}
  all isogeny factors of $J$ occur this way.

  In order to find a field of definition, we can determine the matrix $T$
  corresponding to $e$ and recognize its entries as algebraic numbers. Then
  \cite{KaniRosen1989} shows that the image of the projection $T$ is still
  polarized, and defined over the corresponding field.
\end{proof}

\subsection{Computing symplectic isomorphisms}\label{sec:isos}

When $g_1 = g_2$, Lemma~\ref{L:homcomp} allows us to recover possible
isomorphisms between $J_1$ and $J_2$, as these correspond to the matrices $R$
with $\det(R) = \pm1$.

In particular, this gives us a description of the automorphism group of a
Jacobian variety $J$ as the subgroup of elements of $\End (J)$ with determinant
$1$. This group can be infinite. However, note that we have principal
polarizations on $J_1$ and $J_2$. We take symplectic bases for the homology of
both Jacobians, and let $\alpha : J_1 \to J_2$ be an isomorphism, represented
by $R\in M_{2g,2g}(\ZZ)$.

\begin{definition}\label{def:sympaut}
  We say that $\alpha$ is \defi{symplectic} if we have $R^t E_2 R = E_1$.
\end{definition}

\begin{remark}
  More intrinsically, the definition demands that the canonical intersection
  pairings $E_1$ and $E_2$ on on $\HH_1 (C_1,\ZZ)$ and $\HH_1 (C_2,\ZZ)$
  satisfy $\alpha^* E_2 = E_1$.
\end{remark}
The symplectic automorphisms of $J$ form a group, which is called the
\defi{symplectic automorphism group} $\Aut (J, E)$ of the principally polarized
abelian variety $(J, E)$.

\begin{theorem}\label{thm:sympaut}
  Suppose that $C$ is a smooth curve of genus at least $2$. Then we have the
  following.
  \begin{enumerate}
    \item The symplectic automorphism group of $J$ is finite.
    \item There is a canonical map $\Aut (C) \to \Aut (J, E)$. If $C$ is
      non-hyperelliptic, then this map is an isomorphism; otherwise it induces
      an isomorphism $\Aut (C) \stackrel{\sim}{\to} \Aut (J, E) / \langle -1
      \rangle$.
  \end{enumerate}
\end{theorem}

\begin{proof}
  Part (i) is \cite[5.1.9]{birkenhake-lange}, and (ii) is the Torelli theorem
  \cite[Theorem~12.1]{Milne1986}.
\end{proof}

This shows we can recover $\Aut(C)$ from $\Aut(J,E)$. In fact, from the linear
action of the symplectic automorphism on $H^0(C,\Omega^1_C)^*$ we can recover
its action on a canonical model of $C$ in $\PP H^0(C,\Omega^1_C)^*$. For
non-hyperelliptic curves this realizes the isomorphism $\Aut (J, E) / \langle
-1\rangle\simeq \Aut(C)$ explicitly. For hyperelliptic curves it recovers the
\emph{reduced} automorphism group, which can in fact be determined more
efficiently by purely algebraic methods, as described in \cite{lrs-ants}.

If $C$ is defined over $\Qbar$, then we can verify that the numerical
automorphisms thus obtained are correct by working purely algebraically: by
Proposition~\ref{prop:homos}(iv) we obtain an algebraic expression for $T$. We
can then check by exact calculation that it fixes the defining ideal of the
canonical embedding of $C$.

More generally, given two Jacobians $J_1$ and $J_2$, we can determine the
numerical symplectic isomorphisms between them. To this end, one proceeds as in
the proof of \cite[5.1.8]{birkenhake-lange}: we have
\begin{equation}
  R^t E_2 R = E_1
\end{equation}
or
\begin{equation}\label{eq:sym2}
  (E_1^{-1} R^t E_2) R = 1 .
\end{equation}
In particular, we get
\begin{equation}\label{eq:sym3}
  \tr ((E_1^{-1} R^t E_2) R) = 2 g
\end{equation}
for the common genus $g$ of $C_1$ and $C_2$. Let $B = \left\{ B_1 , \dots , B_d
\right\}$ be a $\ZZ$-basis of $\Hom(J_1, J_2)$. Then we can write
\begin{equation}\label{eq:inbasis}
  R = \sum_{i = 1}^d \lambda_i B_i .
\end{equation}
The positivity of the Rosati involution implies that the set of solutions
$\lambda_1, \dots , \lambda_d$ of \eqref{eq:sym3} is finite. Explicitly, these
can be obtained by using the Fincke-Pohst algorithm \cite{fincke-pohst}. For
the finite set of solutions thus obtained, we check which yield matrices $R$ in
\eqref{eq:inbasis} that numerically satisfy \eqref{eq:sym2}. These matrices
constitute the homology representations $R$ of numerical isomorphisms $J_1 \to
J_2$. From this, we can obtain the corresponding tangent representations $T$ by
Proposition \ref{prop:homos}(iii), and we can verify these algebraically as
above.

\begin{remark}
  Using the same methods, one can determine the maps $C_1 \to C_2$ of a fixed
  degree $d$ by finding the $\alpha$ for which $\alpha_* E_1 = d E_2$. This is
  especially useful if the genus $g_2$ of $C_2$ is larger than $2$, since then
  we can bound $d$ by $(2 g_1 - 2) / (2 g_2 - 2)$.
\end{remark}

In this way, we obtain the following pseudocode.

\smallskip
\noindent\textbf{Algorithm} Compute isomorphisms between curves.

\emph{Input:} Planar equations $f_1, f_2$ for two curves $C_1, C_2$, as well as
a given working precision.

\emph{Output:} A numerical determination of the set of isomorphisms $C_1 \to
C_2$.

\begin{enumerate}
  \item[1.] Check if $g (C_1) = g (C_2)$; if not, return the empty set;
  \item[2.] Check if $C_1$ and $C_2$ are hyperelliptic; if so, use the methods
    in \cite{lrs-ants};
  \item[3.] Otherwise, determine the period matrices $P_1, P_2$ of $C_1, C_2$
    to the given precision, using the algorithm in the introduction;
  \item[4.] Using Lemma \ref{L:homcomp} (see also \cite[6.1]{ants-database}),
    determine a $\ZZ$-basis of $\Hom(J_1, J_2) \subset M_{2g, 2g} (\Z)$
    represented by integral matrices $R \in M_{2g, 2g} (\ZZ)$;
  \item[5.] Using Fincke-Pohst, determine the finite set $S = \left\{ R \in
    \Hom (J_1, J_2) \mid \tr ((E_1^{-1} R^t E_2) R) = 2 g \right\}$;
  \item[6.] Using the canonical morphisms with respect to the chosen bases of
    differentials, return the subset of elements of $S$ that indeed induce an
    isomorphism $C_1 \to C_2$.
\end{enumerate}

\section{Examples}\label{sec:ex}

The examples in this section can be found online at \cite{GitHub}.

\begin{example}\label{E:genus6}
  Consider the curve
  \begin{equation*}
    C : 4 x^6 - 54 x^5 y - 729 x^4 + 108 x^3 y^3 + 39366 x^2 - 54 x y^5 - 531441 .
  \end{equation*}
  This is a non-hyperelliptic curve of genus $6$. Theorem \ref{thm:sympaut}
  shows that, at least numerically, its geometric automorphism group is of
  order $2$ and generated by the involution $\iota : (x, y) \mapsto (-x, -y)$.
  Lemma \ref{L:homcomp} shows that its numerical geometric endomorphism ring is
  of index $6$ in $\Z \times \Z \times \Z$.

  The quotient of $C$ by its automorphism group gives a morphism of degree $2$
  to the genus $2$ curve
  \begin{equation*}\label{eq:D1}
    D_1 : y^2 = x^6 - x^5 + 1 .
  \end{equation*}
  This corresponds to the symmetric idempotent $e_1 = 1 - (1 + \iota)/2$ in the
  endomorphism algebra, whose tangent representation has numerical rank $4$.
  Numerically, there are two other such symmetric idempotents $e_2$, $e_3$.
  Together, their kernels span $\HH_1 (C(\CC), \Z)$, and all of these are of
  dimension $4 = 2 \cdot 2$. This means that along with $A_1 = \Jac (D_1)$
  there should be two other $2$-dimensional abelian subvarieties $A_2$, $A_3$
  of $\Jac (C)$ such that
  \begin{equation*}
    \Jac (C) \sim A_1 \times A_2 \times A_3 .
  \end{equation*}
  We now describe the abelian varieties $A_2$ and $A_3$.

  The tangent representation of an idempotent $e_i$ corresponding to a factor
  $A_i$ has dimension $4$. Its kernel is therefore a subspace $W_i$ of
  $\HH^0(C, \Omega_C^1)$ of dimension $2$. If the idempotent $e_i$ is induced
  by a map of curves $p : C \to D_i$, then $W_i = p^* \HH^0 (D_i,
  \Omega_{D_i}^1)$ for some curve $D_i$ and some projection $p\colon C \to
  D_i$.

  By composing the canonical map with the projection to the projective line $\P
  W_i$, all the idempotents $e_i$ give rise to a cover $C \to \P W_i$. Now if
  $e_i$ is induced by a projection $C \to D_i$ at all, then $D_i$ is a subcover
  of this map $C \to \P W_i$. It turns out that all $e_i$ give rise a subcover
  of the degree $6$ non-Galois cover
  \begin{equation*}
    \begin{aligned}
      C & \to \PP^1 \\
      (x, y) & \to y /x .
    \end{aligned}
  \end{equation*}
  A monodromy calculation gives the Galois closure $Z \to \PP^1$ of this cover:
  its Galois group $G$ is dihedral of order $12$. In particular, considering
  the subgroups of $G$ that properly contain the degree $2$ subgroup
  corresponding to $C \to \PP^1$, we see that there exist exactly two
  non-trivial subcovers $p_1\colon C \to D_1$ and $p_2\colon C \to D_2$ of $C
  \to \PP^1$. These subcovers have degree $2$ and degree $3$, respectively.

  The curves $D_1$ and $D_2$ are both of genus $2$. The first subcover $p_1$ is
  a quotient of $C$ and corresponds to the curve $D_1$ above. The second
  subcover $p_2$ is not a quotient of $C$, but using Galois theory for the
  normal closure still furnishes us with a defining equation of $D_2$, namely
  \begin{equation*}
    D_2\colon y^2 = -16 x^5 - 40 x^4 + 32 x^3 + 88 x^2 - 32 x - 23 .
  \end{equation*}
  We take $A_2$ to be the Jacobian of $D_2$.

  Since we have exhausted all subcovers of the Galois closure $Z \to \PP^1$, we
  conclude that $A_3$ does not arise from a cover $C\to D_3$. Still, using
  analytic methods we find that numerically the subvariety $A_3$ is simple and
  admits a (unique) principal polarization. It is therefore the Jacobian of a
  curve $D_3$ of genus $2$. Calculating the Igusa invariants numerically, we
  reconstruct
  \begin{equation*}
    D_3\colon y^2 = x^6 + 3 x^4 + 3 x^2 + x + 1 .
  \end{equation*}
  We can numerically check that there is a morphism of abelian varieties $\Jac
  (C) \to \Jac (D_3)$ that is compatible with the polarizations on both curves.
  A computation on homology again shows that this morphism cannot come from a
  morphism of curves $C \to D_3$; if it did, the degree of such a morphism
  would have to be $6$, which is impossible in light of the Riemann-Hurwitz
  formula. An explicit correspondence between $C$ and $D_3$ can in principle be
  found by using the methods in \cite{CMSV}; however, this will still be a
  rather involved calculation, which we have therefore not performed yet.
%
\end{example}

\begin{example}\label{E:macbeath}
  Consider the plane model
  \begin{equation*}
    C\colon f (x, y)
        = 1 + 7 x y + 21 x^2 y^2 + 35 x^3 y^3 + 28 x^4 y^4 + 2 x^7 + 2 y^7 = 0
  \end{equation*}
  of the Macbeath curve from \cite{hidalgo-macbeath}, which is due to Bradley
  Brock. Its automorphism group is isomorphic to $\PSL_2 (\F_8)$ and has order
  $504$. We illustrate that the algorithm described in Section~\ref{sec:isos}
  indeed recovers that $\Aut (C)$ is isomorphic to $\PSL_2 (\F_8)$, that $C$ is
  indeed the Macbeath curve, and moreover that all the automorphisms of $C$ are
  already defined over the cyclotomic field $\Q (\zeta_7)$.

  From the adjoint ideal computed by Singular \cite{DGPS} we find a
  $\Q$-rational basis of $7$ global differentials of the form $h \omega$, where
  $\omega = \frac{\partial f}{\partial y} dx$ and where $h$ is one of
  \begin{equation*}
    \begin{aligned}
      \left\{ h_1 ,\dots, h_7 \right\} = \{ 4 x^2 y^2 + 3 x y + 1, 2 y^5 - x^3
      y - x^2, 2 x y^4 + x^4 + y^3, \\ 4 x^2 y^3 + 3 x y^2 + y, 4 x^3 y^2 + 3
    x^2 y + x, 2 x^4 y + y^4 + x^3, 2 x^5 - x y^3 - y^2 \} .
    \end{aligned}
  \end{equation*}
  We can determine a corresponding period matrix to binary precision $100$
  after about a minute's calculation, and find the corresponding numerical
  symplectic automorphism group. It indeed has cardinality $1008$, and its
  elements are well-approximated by relatively simple matrices in the
  cyclotomic field $\Q (\zeta_7)$ that also generate a group $G \subset \GL_7
  (\Q (\zeta_7))$ of order $1008$ with $G \cap \Q (\zeta_7)^* = \langle -1
  \rangle$ and with $G / \langle -1 \rangle \cong \PSL_2 (\F_2^3)$. In practice
  this is of course indication enough that the automorphism group has been
  found.

  To prove this, we choose two elements $T_1, T_2$ of $G$. The first of these
  is the diagonal matrix with entries $\left\{ 1, \zeta_7^2, \zeta_7^4,
  \zeta_7^6, \zeta_7, \zeta_7^3, \zeta_7^5 \right\}$; the other has relatively
  modest entries but is still too large to write down here. We check that these
  matrices generate a subgroup of $G$ of cardinality $504$ that projects
  isomorphically to $G / \langle -1 \rangle$. If we show that $T_1$ and $T_2$
  indeed correspond to automorphisms of $C$, then our claims will be proved,
  since any curve of genus $7$ with (at least) $504$ automorphisms is
  birational to the Macbeath curve.

  To verify this claim, one can use the canonical embedding of $C$ with respect
  to the given basis of global differentials $\left\{ h_i \omega \right\}$.
  Alternatively, one observes that
  \begin{equation*}
    x = h_5 / h_1,\,
    y = h_4 / h_1 .
  \end{equation*}
  This means that after applying one of the transformations $T_1, T_2$ to the
  basis of global differentials to obtain the linear transformations $\left\{
  T_i (h_j \omega) \right\}_j$, we can recover corresponding transformations
  $x'$ and $y'$ in $x$ and $y$ via
  \begin{equation*}
    x' = T_i (h_5 \omega) / T_i (h_1 \omega) ,\,
    y' = T_i (h_4 \omega) / T_i (h_1 \omega) .
  \end{equation*}
  For $T_1$, we get
  \begin{equation*}
    x' = \zeta_7   x ,\,
    y' = \zeta_7^6 y ,
  \end{equation*}
  while for $T_2$ we get two decidedly unpleasant
  rational expressions the degree of whose denominator and numerator both equal
  $5$. In either case, we can check that the corresponding substitutions leave
  the equation for $C$ invariant, which provides us with the desired
  verification of correctness of $T_1$ and $T_2$.
\end{example}

\begin{example}\label{E:prym}
  This example illustrates the value of being able to verify isogeny factors of
  Jacobians numerically. We consider a genus $4$ curve $C$ and an unramified
  double cover $\pi\colon D\to C$. Then $D$ is of genus $7$, and $\Jac(D)$ is
  isogenous to $\Jac(C)\times A$ for some $3$-dimensional abelian variety $A$.
  The theory of Prym varieties shows we can take $A$ to be principally
  polarized. It follows that generally $A$ is a quadratic twist of a Jacobian
  of a genus $3$ curve $F$. In~\cite{Milne1923} W.P.~Milne constructs a plane
  quartic $F$ from a genus $4$ curve $C$ with data that amounts to specifying
  an unramified double cover of $C$. One would guess that $\Jac(F)$ is indeed
  the Prym variety of $D/C$. Here we check this numerically for a particular
  example. A modern, systematic treatment of this construction is in
  preparation \cite{BHS}.

  Let $C$ be the canonical genus $4$ curve in $\PP^3$, described by
  $\Gamma_2=\Gamma_3=0$, where
  \begin{equation*}
    \begin{aligned}
      \Gamma_2&=x^2+xy+y^2+3xz+z^2-yw+w^2,  \Gamma_3&= xyz+xyw+xzw+yzw.
    \end{aligned}
  \end{equation*}
  A plane model for this curve is given by
  \begin{equation*}
    \begin{split}
      \tilde{C}\colon y^4w^2 - y^3w^3 + y^2w^4 + 2y^4w - y^3w^2 + 2yw^4 + y^4 - 2y^2w^2\\
      + yw^3 + w^4 - y^2w - yw^2 + y^2 + 2yw + w^2=0.
    \end{split}
  \end{equation*}
  Since $\Gamma_3$ has four nodal singularities in general position, it is a
  Cayley cubic. It admits a double cover unramified outside the nodes, obtained
  by adjoining the square root of the Hessian of $\Gamma_3$. Since $C$ does not
  pass through the nodes, this induces an unramified double cover $D$ of $C$.
  It is geometrically irreducible and admits a plane model
  \begin{equation*}
    \begin{split}
    \tilde{D}\colon
      u^4v^4 - 3u^4v^2 + u^4 - u^3v^3 - 2u^3v + u^2v^2 - u^2 + 3uv^3 + 2uv + v^4 + v^2
      + 1=0.
    \end{split}
  \end{equation*}
  Milne's construction yields a plane quartic
  \begin{equation*}
    \begin{split}
      F\colon 5s^4 + 28s^3t + 28s^3 + 47s^2t^2 + 76s^2t + 44s^2 + 34st^3 + 82st^2\\
      + 66st + 18s + 16t^4 + 34t^3 + 32t^2 + 18t + 1=0.
    \end{split}
  \end{equation*}
  Numerical computation shows that $\End(\Jac(C))=\ZZ$ and that
  $\End(\Jac(F))=\ZZ$, which can be confirmed by the $\ell$-adic methods in
  \cite{CMSV}.  It follows that $\Hom(\Jac(F),\Jac(C))=0$. Furthermore, we find
  that $\Hom(\Jac(C),\Jac(D))$ and $\Hom(\Jac(F),\Jac(D))$ are $1$-dimensional,
  so it follows that $\Jac(D)\sim\Jac(C)\times\Jac(F)$ and that $\Jac(F)$ lies
  in the Prym variety of the cover $D\to C$. Thus, we obtain numerical evidence
  that Milne indeed provides a construction of a curve $F$ generating the Prym
  variety.
\end{example}

\begin{example} This example is similar to Example~\ref{E:genus6}, except that the curve has extra automorphisms, which allows a more direct identification of the factors.
  Consider the curve
  \begin{equation*}
    C : 3x^5+2x^3y^2-xy^4-9x^4+6x^2y^2-y^4+4=0
  \end{equation*}
  of genus $6$. Computing the canonical map shows that $C$ is
  non-hyperelliptic. Theorem \ref{thm:sympaut} shows numerically that 
  $\Jac (C)$ has a symplectic automorphism group that is dihedral of order
  $12$, which means the automorphism group of $C$ is isomorphic to $\Sym(3)$. From the tangent representation one can find that these numerically determined automorphisms act linearly on the affine model of $C$, and it is easy to recover exact representatives and verify that they are indeed automorphisms of $C$.

  We compute the quotient of $C$ by each of the involutions. For the involution $(x,y)\mapsto(x,-y)$ this is particularly straightforward.
  For the other involution $(x,y)\mapsto(-\frac{1}{2}x+\frac{i}{2}y,-\frac{3i}{2}x+\frac{1}{2}y)$ and its quadratic conjugate it is slightly more work, but in all cases we find that the quotient is birational to
  \[
   D_1 : y^2 = x^6 - x^5 + 1 .
  \]
  We can compute numerically that the $\ZZ$-module of homomorphisms from $\Jac(D_1)$ to $\Jac(C)$ is of rank $2$. The images span a rank $8$ submodule on the homology, indicating that $\Jac(D_1)^2$ is an isogeny factor of $\Jac(C)$.

  Similarly, the quotient of $C$ by the order $3$ group generated by $(x,y)\mapsto(-\frac{1}{2}x+\frac{i}{2}y,\frac{3i}{2}x-\frac{1}{2}y)$ yields
  \[
      D_2 : y^2 = 2 x^5 + 27 x^4 - 54 x^2 + 27.
  \]
  This provides us with sufficient information to conclude that
  \[
  \Jac(C)\sim \Jac(D_1)^2\times \Jac(D_2).
  \]
  Both of
  the curves $D_1$ and $D_2$ have endomorphism ring $\ZZ$ over $\Qbar$, as can
  be verified by the $\ell$-adic methods in \cite{CMSV}. Similar considerations
  show that the Jacobians of $D_1$ and $D_2$ are not isogenous. Therefore the
  numerical endomorphism ring of $\Jac (C)$ is an order in $\Q \times M_2
  (\Q)$; closer consideration of the representation of the action of its
  generators on homology obtained by Lemma \ref{L:homcomp} shows that it is of
  index $9$ in $\Z \times M_2 (\Z)$.
%
\end{example}

\bibliographystyle{abbrv}

\begin{bibdiv}
\begin{biblist}

\bib{ACGH1985}{book}{
      author={Arbarello, E.},
      author={Cornalba, M.},
      author={Griffiths, P.~A.},
      author={Harris, J.},
       title={Geometry of algebraic curves. {V}ol. {I}},
      series={Grundlehren der Mathematischen Wissenschaften},
   publisher={Springer-Verlag, New York},
        date={1985},
      volume={267},
        ISBN={0-387-90997-4},
         url={https://doi.org/10.1007/978-1-4757-5323-3},
}

\bib{Aurenhammer1991}{article}{
      author={Aurenhammer, Franz},
       title={Voronoi diagrams--a survey of a fundamental geometric data
  structure},
        date={1991-09},
        ISSN={0360-0300},
     journal={ACM Comput. Surv.},
      volume={23},
      number={3},
       pages={345\ndash 405},
         url={http://doi.acm.org/10.1145/116873.116880},
}

\bib{BaileyEA2005}{article}{
      author={Bailey, David~H.},
      author={Jeyabalan, Karthik},
      author={Li, Xiaoye~S.},
       title={A comparison of three high-precision quadrature schemes},
        date={2005},
        ISSN={1058-6458},
     journal={Experiment. Math.},
      volume={14},
      number={3},
       pages={317\ndash 329},
         url={http://projecteuclid.org/euclid.em/1128371757},
}

\bib{Baker1893}{article}{
      author={{Baker}, H.~F.},
       title={{Examples of the application of Newton's polygon to the theory of
  singular points of algebraic functions}},
        date={1893},
     journal={Transactions of the Cambridge Philosophical Society},
      volume={15},
       pages={403},
}

\bib{birkenhake-lange}{book}{
      author={Birkenhake, Christina},
      author={Lange, Herbert},
       title={Complex abelian varieties},
     edition={Second},
      series={Grundlehren der Mathematischen Wissenschaften},
   publisher={Springer-Verlag, Berlin},
        date={2004},
      volume={302},
        ISBN={3-540-20488-1},
         url={https://doi.org/10.1007/978-3-662-06307-1},
}

\bib{bogaert}{article}{
      author={Bogaert, I.},
       title={Iteration-free computation of {G}auss-{L}egendre quadrature nodes
  and weights},
        date={2014},
        ISSN={1064-8275},
     journal={SIAM J. Sci. Comput.},
      volume={36},
      number={3},
       pages={A1008\ndash A1026},
         url={https://doi.org/10.1137/140954969},
}

\bib{ants-database}{article}{
      author={Booker, Andrew~R.},
      author={Sijsling, Jeroen},
      author={Sutherland, Andrew~V.},
      author={Voight, John},
      author={Yasaki, Dan},
       title={A database of genus-2 curves over the rational numbers},
        date={2016},
        ISSN={1461-1570},
     journal={LMS J. Comput. Math.},
      volume={19},
      number={suppl. A},
       pages={235\ndash 254},
         url={https://doi.org/10.1112/S146115701600019X},
}

\bib{BHS}{unpublished}{
      author={Bruin, Nils},
      author={Sertoz, Emre},
       title={Prym varieties of genus 4 curves},
        date={2018},
        note={In preparation},
}

\bib{GitHub}{unpublished}{
      author={Bruin, Nils},
      author={Sijsling, Jeroen},
      author={Zotine, Alexandre},
       title={Calculations with numerical {J}acobians},
        date={2018},
  note={\href{https://github.com/nbruin/examplesNumericalEndomorphisms}{\tt
  https://github.com/nbruin/examplesNumericalEndomorphisms}},
}

\bib{CMSV}{unpublished}{
      author={Costa, Edgar},
      author={Mascot, Nicolas},
      author={Sijsling, Jeroen},
      author={Voight, John},
       title={Rigorous computation of the endomorphism ring of a {J}acobian},
        date={2016},
        note={\href{http://arxiv.org/abs/1705.09248}{\tt arXiv:1707.01158}},
}

\bib{DGPS}{misc}{
      author={Decker, Wolfram},
      author={Greuel, Gert-Martin},
      author={Pfister, Gerhard},
      author={Sch\"onemann, Hans},
       title={{\sc Singular} {4-1-1} --- {A} computer algebra system for
  polynomial computations},
         how={\url{http://www.singular.uni-kl.de}},
        date={2018},
}

\bib{DeconinckVanHoeij2001}{article}{
      author={Deconinck, Bernard},
      author={van Hoeij, Mark},
       title={Computing {R}iemann matrices of algebraic curves},
        date={2001},
        ISSN={0167-2789},
     journal={Phys. D},
      volume={152/153},
       pages={28\ndash 46},
         url={https://doi.org/10.1016/S0167-2789(01)00156-7},
        note={Advances in nonlinear mathematics and science},
}

\bib{fincke-pohst}{article}{
      author={Fincke, U.},
      author={Pohst, M.},
       title={Improved methods for calculating vectors of short length in a
  lattice, including a complexity analysis},
        date={1985},
        ISSN={0025-5718},
     journal={Math. Comp.},
      volume={44},
      number={170},
       pages={463\ndash 471},
         url={https://doi.org/10.2307/2007966},
}

\bib{Frobenius1879}{article}{
      author={Frobenius, G.},
       title={Theorie der linearen {F}ormen mit ganzen {C}oefficienten},
        date={1879},
        ISSN={0075-4102},
     journal={Crelle},
      volume={86},
       pages={146\ndash 208},
         url={https://doi.org/10.1515/crll.1879.86.146},
}

\bib{Hess2004}{incollection}{
      author={Hess, F.},
       title={An algorithm for computing isomorphisms of algebraic function
  fields},
        date={2004},
   booktitle={Algorithmic number theory},
      series={Lecture Notes in Comput. Sci.},
      volume={3076},
   publisher={Springer, Berlin},
       pages={263\ndash 271},
         url={https://doi.org/10.1007/978-3-540-24847-7_19},
}

\bib{hidalgo-macbeath}{unpublished}{
      author={Hidalgo, Ruben~A.},
       title={About the {F}ricke-{M}acbeath curve},
        date={2017},
        note={\href{https://arxiv.org/abs/1703.01869}{\tt arXiv:1703.01869}},
}

\bib{johansson2018}{unpublished}{
      author={Johansson, Frederik},
       title={Numerical integration in arbitrary-precision ball arithmetic},
        date={2018},
        note={\href{https://arxiv.org/abs/1802.07942}{\tt arXiv:1802.07942}},
}

\bib{KaniRosen1989}{article}{
      author={Kani, E.},
      author={Rosen, M.},
       title={Idempotent relations and factors of {J}acobians},
        date={1989},
        ISSN={0025-5831},
     journal={Math. Ann.},
      volume={284},
      number={2},
       pages={307\ndash 327},
         url={https://doi.org/10.1007/BF01442878},
}

\bib{Kranich2016}{unpublished}{
      author={Kranich, Stefan},
       title={An epsilon-delta bound for plane algebraic curves and its use for
  certified homotopy continuation of systems of plane algebraic curves},
        date={2015},
        note={\href{https://arxiv.org/abs/1505.03432}{\tt arXiv:1505.03432}},
}

\bib{lrs-ants}{incollection}{
      author={Lercier, Reynald},
      author={Ritzenthaler, Christophe},
      author={Sijsling, Jeroen},
       title={Fast computation of isomorphisms of hyperelliptic curves and
  explicit {G}alois descent},
        date={2013},
   booktitle={A{NTS} {X}---{P}roceedings of the {T}enth {A}lgorithmic {N}umber
  {T}heory {S}ymposium},
      series={Open Book Ser.},
      volume={1},
   publisher={Math. Sci. Publ., Berkeley, CA},
       pages={463\ndash 486},
         url={https://doi.org/10.2140/obs.2013.1.463},
}

\bib{Milne1986}{incollection}{
      author={Milne, J.~S.},
       title={Jacobian varieties},
        date={1986},
   booktitle={Arithmetic geometry ({S}torrs, {C}onn., 1984)},
   publisher={Springer, New York},
       pages={167\ndash 212},
}

\bib{Milne1923}{article}{
      author={Milne, W.~P.},
       title={Sextactic cones and tritangent planes of the same system of a
  quadri-cubic curve},
        date={1923},
     journal={Proceedings of the London Mathematical Society},
      volume={s2-21},
      number={1},
       pages={373\ndash 380},
}

\bib{molin-neurohr}{unpublished}{
      author={Molin, Pascal},
      author={Neurohr, Christian},
       title={Computing period matrices and the {A}bel-{J}acobi map of
  superelliptic curves},
        date={2017},
        note={\href{http://arxiv.org/abs/1707.07249}{\tt arXiv:1707.07249}},
}

\bib{neurohr-thesis}{thesis}{
      author={Neurohr, Christian},
       title={Efficient integration on {R}iemann surfaces and applications},
        type={Ph.D. Thesis},
        date={2018},
}

\bib{sertoz2018}{unpublished}{
      author={Sert\"oz, Emre~Can},
       title={Computing periods of hypersurfaces},
        date={2018},
        note={\href{https://arxiv.org/abs/1803.08068}{\tt arXiv:1803.08068}},
}

\bib{swier}{unpublished}{
      author={Swierczewski, Chris},
       title={abelfunctions: A library for computing with {A}belian functions,
  {R}iemann surfaces, and algebraic curves},
        date={2017},
        note={\href{https://github.com/abelfunctions/abelfunctions}{\tt
  https://github.com/abelfunctions/abelfunctions}},
}

\bib{TretkoffTretkoff1984}{incollection}{
      author={Tretkoff, C.~L.},
      author={Tretkoff, M.~D.},
       title={Combinatorial group theory, {R}iemann surfaces and differential
  equations},
        date={1984},
   booktitle={Contributions to group theory},
      series={Contemp. Math.},
      volume={33},
   publisher={Amer. Math. Soc., Providence, RI},
       pages={467\ndash 519},
         url={https://doi.org/10.1090/conm/033/767125},
}

\bib{vanWamelen2006}{incollection}{
      author={van Wamelen, Paul~B.},
       title={Computing with the analytic {J}acobian of a genus 2 curve},
        date={2006},
   booktitle={Discovering mathematics with {M}agma},
      series={Algorithms Comput. Math.},
      volume={19},
   publisher={Springer, Berlin},
       pages={117\ndash 135},
         url={https://doi.org/10.1007/978-3-540-37634-7_5},
}

\end{biblist}
\end{bibdiv}

\appendix
\section{Guide to implementation}

We briefly describe how the implementation of our routines in \SageMath\ can be
used. For the sake of brevity, we limit ourselves to genus $1$ and $2$ curves
here, but obviously our implementations mainly provide novel capabilities for
higher genus curves.

The basic data structure is most easily accessed by defining a plane algebraic
curve. The most important attribute at this point is specifying the numerical
working precision (in bits).
\begin{verbatim}
sage: A2.<x,y>=AffineSpace(QQ,2)
sage: E=Curve(x^3+y^3+1)
sage: SE=E.riemann_surface(prec=60)
\end{verbatim}
The various attributes and routines of the complex torus derived from the
Riemann surface \texttt{SE} can be accessed via methods on the object. For
instance, we can (numerically) compute a $\ZZ$-basis for the homomorphisms,
represented by matrices $R$ acting on the homology. There are parameters
available to tune the use of LLL to recognize actual endomorphisms, but some
care is taken to choose reasonable defaults based on the precision available.
\begin{verbatim}
sage: Rs=SE.symplectic_isomorphisms(SE);Rs
[ [ 0 -1]  [ 1  1]  [1 0]  [ 0  1]  [-1 -1]  [-1  0]
  [ 1  1], [-1  0], [0 1], [-1 -1], [ 1  0], [ 0 -1] ]
\end{verbatim}
In simple cases such as this one we can recognize the algebraic matrices acting
on the tangent space (\SageMath\ here chooses to express things in a sixth root
of unity).
\begin{verbatim}
sage: Ts=SE.tangent_representation_algebraic(A);Ts
[[a], [-a + 1], [1], [-a], [a - 1], [-1]]
sage: Ts[0]^3
[-1]
\end{verbatim}
The Jacobian of the following genus $2$ curve has $E$ as an isogeny factor.
\begin{verbatim}
sage: C=HyperellipticCurve(X^6 + 3*X^4 + 3*X^2 + 2)
sage: SC=C.riemann_surface(prec=60)
sage: SC.homomorphism_basis(SE)
[ [ 0  0  1 -1]  [-1  1  0  0]
  [-1  1 -1  1], [ 0  0 -1  1] ]
\end{verbatim}
We can identify the cofactor as well.
\begin{verbatim}
sage: SE2=EllipticCurve([1,0]).riemann_surface(prec=60)
sage: SC.homomorphism_basis(SE2)
[ [ 1  1 -1 -1]  [ 0  0  1  1]
  [ 0  0  1  1], [-2 -2  0  0] ]
\end{verbatim}

\end{document}